\documentclass[11pt]{amsart}
\usepackage{amsmath}
\usepackage{amsthm}
\usepackage{amsbsy}
\usepackage{amssymb}
\usepackage{amsfonts}
\usepackage[utf8]{inputenc}
\usepackage[backend=biber]{biblatex}
\usepackage{hyperref}
\hypersetup{colorlinks=true}
\usepackage[capitalise]{cleveref}
\usepackage{pgf,tikz,pgfplots}
\pgfplotsset{compat=1.14}
\usepackage{mathrsfs}
\usetikzlibrary{arrows}
\usepackage{microtype}
\author{Hendrik Heine}
\title{Path spaces I: A Menger-type result}
\bibliography{menger_ps}
\newtheorem{thm}{Theorem}[section]
\Crefname{thm}{Theorem}{Theorems}
\newtheorem{lem}[thm]{Lemma}
\Crefname{lem}{Lemma}{Lemmas}
\newtheorem{prop}[thm]{Proposition}
\Crefname{prop}{Proposition}{Propositions}

\Crefname{cor}{Corollary}{Corollaries} 

\Crefname{que}{Question}{Questions}

\Crefname{con}{Conjecture}{Conjectures}

\Crefname{clm}{Claim}{Claims}

\Crefname{goal}{Goal}{Goals}

\theoremstyle{definition}

\Crefname{defn}{Definition}{Definitions}
\newtheorem{ex}[thm]{Example}
\Crefname{ex}{Example}{Examples}

\Crefname{rem}{Remark}{Remarks}
\begin{document}
\begin{abstract}
Infinite graphs are finitary in the sense that their points are connected via finite paths. So what would an infinitary generalization of finite graphs look like? Usually this question is answered with the aid of topology, e.g. in the case of graph-like spaces. Here we introduce a more combinatorial answer, which we call path space, and prove a version of Menger's theorem for it. Since there are many topological path-like objects which induce path spaces, this result can be applied in a variety of settings.
\end{abstract}
\maketitle
\section{Introduction}
Paths in graphs, being the basic notion underlying connectivity, have been generalized in various ways, the most prominent being the notion of topological arcs. This connection has become more important in recent years with the study of topological infinite graphs. Considering an infinite graph as a $1$-complex with ends as points at infinity, topological arcs in this space can substitute for paths of finite graphs in a number of theorems. One of the original successes of this technique was the cycle space duality for locally finite graphs described by Diestel and Kühn in \cite{DK:Cycles}.
With this approach becoming more prolific, this also generated additional interest in a superclass of these spaces, so called graph-like spaces. These being topological spaces, arcs could be studied as before and in \Cite{TV:Menger} Thomassen and Vella investigated the extent to which Menger's theorem held in spaces like these. 

However, arcs do not necessarily capture all 'path-like' structures in these spaces. By arraying edges like $\omega_1$, one can construct a graph-like space called the long-line. Even though this space appears to consist of just a single path, there is no arc between its endpoints since $[0,1]$ has no uncountable ascending sequences. This can be resolved using a notion called pseudo-line introduced in \cite{BCC:InfGraphMatroids}, which can be born from any linear order.

Although not every pseudo-arc is an arc and not every Hausdorff space is graph-like many proofs for arcs in Hausdorff spaces and pseudo-arcs in graph-like spaces are parallel. This raises the question of whether there could be a further generalization of both which would allow us to prove theorems valid for any notion of paths sharing some essential properties.

In this article we will describe such a notion, which we call path spaces. These are sets of linear orders satisfying some compatibility axioms. By forgetting all information about the paths except their induced linear orders, we can forego the need for any setting for our paths to live in and define these spaces purely by the interactions of the paths themselves.

Path spaces can also be viewed as a different way to generalize the concept of finite graph to the infinite realm. While infinite graphs are still finitary, that is all paths are finite, path spaces allow for true infinitary connectivity. Indeed, graphs correspond exactly to those path spaces all whose paths are finite. In addition, many of the standard results about connectivity in graphs generalize to path spaces, as this series of articles will demonstrate.

The first theorem we will prove about path spaces will be a version of Menger's theorem. This being perhaps the fundamental theorem of graph connectivity and spawning a multitude of variants, strenghthenings and generalizations, including the breakthrough result in \cite{AB:Menger} by Aharoni and Berger, versions of it were already proven for many of the structures mentioned, including topological graphs in \cite{Diestel:EndMenger} and arcs in topological spaces (or graph-like spaces) e.g. in \cite{TV:Menger} and \cite{Gilbert:Menger}, but our version extends further, e.g. to long arcs.

Building hereupon, further articles (\cite{Heine:C2} and \cite{Heine:C3}) will consider decompositions of path spaces similar to tree decompositions and how these allow us to separate path spaces into blocks or $3$-connected parts. Furthermore, in \cite{Heine:Ubi} we will prove some ubiquity results for path spaces.

This article starts with a section defining path spaces and developing some basic notions and results. This is followed by a discussion of examples. Afterwards we observe some properties of alternating walks in path spaces and use them to prove the main result. The final section deals with counterexamples, in particular we give an example of a graph-like space which fails to satisfy Menger's theorem for cardinality $\aleph_0$.

\section{Path spaces}
Let us start by defining some preliminary notions to prepare the introduction of path spaces.
Call a subset $Y$ of a linearly ordered set $X$ \emph{complete}, if for any nonempty $Z \subseteq Y$ there exists a supremum and infimum in $X$ and these are contained in $Y$.
In this paper, a \emph{path} will be a complete linearly ordered set. A set of paths $\mathcal{P}$ is called \emph{compatible} if for any $P,Q \in \mathcal{P}$ the set $P \cap Q$ is complete in $P$. For a path $P$ with $x \leq y$ in $P$ we call the interval from $x$ and $y$ with the induced order its \emph{segment} from $x$ to $y$. Given two paths $P$ and $Q$ we say that $P$ \emph{connects to} $Q$, if  $P \cap Q = \{x\}$ where $x$ is the maximum of $P$ and the minimum of $Q$. If this is so, we call the union of $P$ and $Q$ with the induced order their \emph{concatenation}. The \emph{inverse} of a path is obtained by reversing its order.

Now we have enough to define path spaces. Since the bookkeeping involved in the proof of Menger's theorem is more natural for directed objects, we also introduce a directed version of path spaces. 
A set of paths is called a \emph{dipath space} if is compatible and closed under segments and concatenations. A dipath space is a \emph{path space} if it is also closed under inverses.
The \emph{ground set} $V(\mathcal{P})$ of a dipath space $\mathcal{P}$ is the union of the ground sets of its paths.

Let $\mathcal{P}$ be a set of paths. We can closed it under segments by just adding all segments of its paths and close it under concatenations by inductively adding in countably many steps all concatenations of paths constructed so far.
Now we define $\widehat{P}$ by closing under segments and then closing under concatenations.
\begin{lem}
Let $\mathcal{P}$ be a compatible set of paths. Then $\widehat{\mathcal{P}}$ is a dipath space.
\end{lem}
\begin{proof}
Let us first show that $\widehat{\mathcal{P}}$ is closed under segments. Let $Q$ be a segment of some $P \in \mathcal{P}$. Now $P$ is the concatenation of finitely many $P_1, \dots, P_n$ which are segments of paths in $\mathcal{P}$. We may assume w.l.o.g. that $Q$ meets all of these in a nontrivial segment, otherwise we may move to a shorter concatenation. Let $Q_1$ be the segment of $P_1$ which $Q$ meets and define $Q_n$ similary. Then $Q$ is the concatenation of $Q_1, P_2, \dots, P_{n-1}, Q_n$, which are all segments of paths in $\mathcal{P}$ and thus contained in $\hat{\mathcal{P}}$.

Now it remains to show that $\widehat{\mathcal{P}}$ is compatible. Since segments are closed under suprema and infima closing a compatible set under segments keeps it compatible.  Thus it suffices to prove that adding the concatenation of two paths from a compatible set to this set leaves it compatible. For this let $P_1, P_2, Q$ be three paths from a compatible set and let $P$ be the concatenation of $P_1$ and $P_2$. Let $Z \subseteq P \cap Q$ be nonempty. Clearly there is a supremum of $Z$ in $Q$ as required, namely the maximum of the suprema of $Z \cap P_1$ and $Z \cap P_2$ in $Q$. Furthermore if $Z$ does not meet $P_2$, then the supremum of $Z$ in $P_1$ is the supremum of $Z$ in $P$ and otherwise the supremum of $Z$ in $P_2$ is the supremum of $Z$ in $P$. For infima the proof proceeds analogously.
\end{proof}
This justifies calling $\widehat{\mathcal{P}}$ the \emph{induced dipath space} of $\mathcal{P}$.
If we close $\mathcal{P}$ under inverses before taking the induced dipath space then we obtain a path space $\overline{\mathcal{P}}$, which we call the \emph{induced path space} of $\mathcal{P}$.

Given a path space $\mathcal{P}$ we write $x \sim y$ for $x,y$ in its ground set if there exists some $P \in \mathcal{P}$ with minimum $x$ and maximum $y$.
\begin{lem}
$\sim$ is an equivalence relation.
\end{lem}
\begin{proof}
It is clearly reflexive and symmetric, so let $x,y,z \in V(\mathcal{P})$ with $x \sim y$ and $y \sim z$ be given and let $P$ and $Q$ be paths witnessing this respectively. Since $P \cap Q$ is complete in $Q$, it has a maximum $m$. Let $P'$ be the segment of $P$ up to $m$ and $Q'$ the segment of $Q$ starting from $m$. Then $P'$ connects to $Q'$ and their concatenation witnesses $x \sim z$.
\end{proof}
We call the equivalence classes defined by $\sim$ \emph{components} of $\mathcal{P}$ and call $\mathcal{P}$ \emph{connected} if it has just one component.
In general dipath spaces we define components and connectedness via the induced path space.

We may delete a set of paths $\mathcal{Q}$ from a dipath space $\mathcal{P}$ by removing all those paths which share a nontrivial segment with a path from $\mathcal{Q}$. This also gives a dipath space.

Given a path space $\mathcal{P}$ and some $v \in V(\mathcal{P})$ we can define a relation on the nontrivial paths of $\mathcal{P}$ starting at $v$ where two paths are equivalent if there exists a nontrivial path starting at $v$ in $\mathcal{P}$ which is an initial segment of both. Clearly, this an equivalence relation. We call its classes the \emph{outdirections} at $v$. The number of these directions is the \emph{outdegree} of $v$. Similarly we can define \emph{indirections} and the \emph{indegree} of $v$.

Let $\mathcal{P}$ be a dipath space. A \emph{walk} in $\mathcal{P}$ is a map $f: P \rightarrow V(\mathcal{P})$ for some path $P$ not neccessarily from $\mathcal{P}$ such that for any $x \in P$ and small enough nontrivial closed interval $I$ in $P$ beginning or ending in $x$ we have $f[I] \in \mathcal{P}$. The set of paths of $\mathcal{P}$ which occurs as $f[I]$ for any closed interval $I$ is called the \emph{space induced by $f$}.

Note that even if $\mathcal{P}$ is the set of arcs of a Hausdorff space, say, its walks in this sense do not necessarily match the paths of the topological space. Indeed, this is impossible since there may be multiple spaces with different sets of paths, but the same set of arcs.\footnote{Indeed, one may construct a topological star either such that there is a topological path reaching each leaf and converging to the center or such that every path reaches only finitely many leafs.} However, the single notion of walk just defined will enough for our purposes.

A path space $\mathcal{P}$ is called \emph{finitary} if there exists a set of finitely many paths inducing $\mathcal{P}$. 
\begin{lem} \label{walkfin}
	For any walk $f$ in some dipath space $\mathcal{P}$ the space induced by $f$ in  $\mathcal{P}$ is finitary and connected.
\end{lem}
\begin{proof}
	Let $f: P \rightarrow V(\mathcal{P})$ be a walk, let $a$ be the minimum and $b$ the maximum of $P$. Now we define $c$ to be the supremum of all $x$ such that the space induced by $f\restriction_{[a,x]}$ is finitary and connected. Let $I$ be a nontrivial closed interval ending at $c$ small enough that $f[I] \in \mathcal{P}$ and let $J$ be similar beginning at $c$. Now $f[I]$ clearly witnesses that the space induced by $f\restriction_{[a,c]}$ is connected. But it is also finitary, since $I$ contains some $x$ for which the space induced by $f\restriction_{[a,x]}$ is finitary. Thus $c$ is a maximum. This, however, implies that $c=b$, since otherwise any point of $J$ larger than $c$ would clearly witness that $c$ is not an upper bound.
\end{proof}
It follows that the space induced by $f$ is the same for any $\mathcal{P}$ in which $f$ is a walk. Thus we will just write $\widehat{f}$ for the space induced by $f$. A walk which has no two nontrivial segments which are have the same or the inverse image is called a \emph{trail}.
\section{Examples of path spaces}
In this section we will give various examples of (di)path spaces. Since the proofs are all straightforward and very similar, we will give just one and omit the rest.

Let $X$ be a Hausdorff space and $\Phi: [0,1] \rightarrow X$ an arc in $X$. Then $\Phi$ induces a path $P_\Phi$ on its image given by the standard ordering on $[0,1]$. Let $A(X)$ be the set of all these $P_\Phi$ for all arcs $\Phi$ in $X$ (together with the trivial path for each point).
\begin{lem} \label{toppathsys}
	$A(X)$ is a path space.
\end{lem}
\begin{proof}
	Clearly $A(X)$ is closed under inverses.
	
	Let $P_\Phi \in A(X)$ be arbitrary and let $a,b$ be two points in the image of $f$, w.l.o.g. different. Then by scaling the interval between their preimages, we can obtain a new arc $\Psi$ between $a$ and $b$ such that $P_\Psi$ is a segment of $P_\Phi$.
	
	Let $P_\Phi$ and $P_\Psi$ in $A(X)$ be such that the maximal element of $P_\Phi$ is the minimal element of $P_\Psi$, but otherwise disjoint. By scaling and combining $\Phi$ and $\Psi$ we then obtain a new arc $\Lambda$, such that $P_\Lambda$ is the concatenation of $P_\Phi$ and $P_\Psi$.
	
	Now let $P_\Phi$ and $P_\Psi$ in $A(X)$ be arbitrary. Since the images of $\Phi$ and $\Psi$ are compact in $X$, so is their intersection. Since $X$ is Hausdorff, $P_\Phi \cap P_\Psi$ is closed. Because any closed set in $[0,1]$ is complete, so is $P_\Phi \cap P_\Psi$.
\end{proof}
In a Hausdorff space $X$ one could also consider the set of injective maps from any compact, connected LOTS to $X$ and obtain a path space the same way. 

Let $G$ be a graph-like space and $L$ a pseudo-line in $G$. Then $f$ together with an endvertex $v$ of $L$ induces a path $Q_L^v$ on $V(L)$ given by the order of $L$ with minimal element $v$. Let $P(G)$ be the set of all these $Q_L^v$. This is a path space.

Furthermore given some set $o$ of orientations of edges of $G$, we can define $P^o(G)$ as the set of all $Q_L^v$ such that the orientation towards $v$ of every edge of $L$ is contained in $o$. This is a dipath space.
As a corollary the set of paths in a graphs defines a path space and the set of dipaths in a digraph defines a dipath space. Conversely, given a (di)pathspace $\mathcal{P}$ with all paths finite one can extract a set of (directed) edges by taking segments and these define a (di)graph inducing  $\mathcal{P}$.

In both the case of LOTS and pseudo-lines one may choose some infinite cardinal as a maximum length for these and still obtain a (di)path space.

Let $k \in \{1, 2, \dots, \infty\}$. If $M$ is a differentiable manifold and $f$ an injective, piecewise $k$-times (continously) differentiable curve on $M$, let $Q_f$ be the image of $f$ with the order induced by $[0,1]$. Then the set of all these $Q_f$ forms a path space.

All the classes of path spaces considered so far are in some sense topological. We will now look at a path space which does not belong to any of them and is slightly pathological.
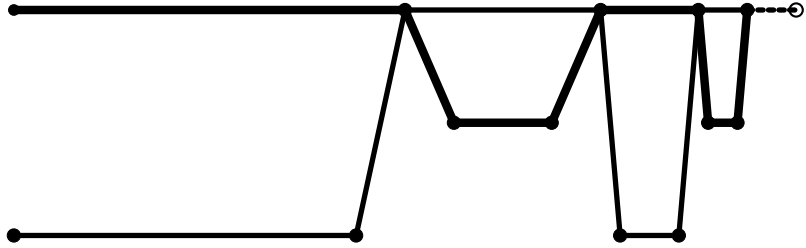
\begin{figure}
\centering
\begin{tikzpicture}[line cap=round,line join=round,>=triangle 45,x=1.3cm,y=1.5cm]
\clip(-0.5,-2.5) rectangle (8.5,0.5);
\draw [line width=3.2pt] (0,0)-- (4,0);
\draw [line width=2pt] (4,0)-- (6,0);
\draw [line width=3.2pt] (6,0)-- (7,0);
\draw [line width=2pt] (7,0)-- (7.5,0);
\draw [line width=3.2pt] (4,0)-- (4.5,-1);
\draw [line width=3.2pt] (4.5,-1)-- (5.5,-1);
\draw [line width=3.2pt] (5.5,-1)-- (6,0);
\draw [line width=2pt] (0,-2)-- (3.5,-2);
\draw [line width=2pt] (3.5,-2)-- (4,0);
\draw [line width=2pt,dotted] (7.5,0)-- (8,0);
\draw [line width=2pt] (6,0)-- (6.2,-2);
\draw [line width=2pt] (6.2,-2)-- (6.8,-2);
\draw [line width=2pt] (6.8,-2)-- (7,0);
\draw [line width=3.2pt] (7,0)-- (7.1,-1);
\draw [line width=3.2pt] (7.1,-1)-- (7.4,-1);
\draw [line width=3.2pt] (7.4,-1)-- (7.5,0);
\begin{scriptsize}
\draw [fill=black] (0,0) circle (2pt);
\draw [fill=black] (4,0) circle (2.5pt);
\draw [fill=black] (6,0) circle (2.5pt);
\draw [fill=black] (7,0) circle (2.5pt);
\draw [fill=black] (7.5,0) circle (2.5pt);
\draw [fill=black] (4.5,-1) circle (2.5pt);
\draw [fill=black] (5.5,-1) circle (2.5pt);
\draw [fill=black] (0,-2) circle (2.5pt);
\draw [fill=black] (3.5,-2) circle (2.5pt);
\draw [color=black,thick] (8,0) circle (2.5pt);
\draw [fill=black] (6.2,-2) circle (2.5pt);
\draw [fill=black] (6.8,-2) circle (2.5pt);
\draw [fill=black] (7.1,-1) circle (2.5pt);
\draw [fill=black] (7.4,-1) circle (2.5pt);
\end{scriptsize}
\end{tikzpicture}
\label{fig:exspace}
\caption{The space from \cref{exspace}. The hollow circle signifies that this point does not form a path with the top ray.}
\end{figure}
\begin{ex} \label{exspace}
Let $X = [0,1] \times \{0,1,2\} \setminus \{(1,0)\}$. Let $\sim$ be the relation on $X$ with $(x,a) \sim (y,b)$ if $x=y$ and one of the following conditions holds:
\begin{enumerate}
	\item{$a=b$}
	\item{$\{a,b\} = \{0,1\}$ and $x \in (1-\frac{1}{2^n},1-\frac{1}{2^{n+1}})$ for $n$ even}
	\item{$\{a,b\} = \{0,2\}$ and $x \in (1-\frac{1}{2^n},1-\frac{1}{2^{n+1}})$ for $n$ odd}
	\item{$x = 1-\frac{1}{2^n}$ for some $n$}
	\item{$x=1$}
\end{enumerate} Note that $\sim$ is an equivalence relation. Let $X'$ be its set of equivalence classes and $f: X \rightarrow X'$ the natural surjection. Let $Y$ consist of all paths of the form $[0,x] \times \{0\}$ for $x \in [0,1)$ with the obvious order together with $[0,1] \times \{1\}$ and $[0,1] \times \{2\}$ and let $Y'$ be the set of images of all these paths under $f$. Then $Y'$ is a compatible set of paths and so $\mathcal{P} = \overline{Y'}$ is a path space.
\end{ex}
\section{Menger's theorem}
In this section we prove our main result, namely Menger's theorem for dipath spaces. We closely follow the augmenting paths proof of Menger's theorem in \cite{Diestel:GT}. 

Perhaps surprisingly, the more difficult part of this proof in which the separator is constructed can be generalized with basically no extra work.
In order to emulate the argument that alternating paths actually augment the path system, however, we need to control the components of the relevant symmetric difference.

This will be accomplished by showing that these dipath spaces are \emph{rayless}, that is for any ascending sequence $R$ of paths there exists a finite set of paths whose union contains a final segment of the union of $R$.

The following lemma demonstrates that this will suffice.

\begin{lem} \label{pathcircuit}
	Let $\mathcal{P}$ be a connected, rayless dipath space with every indegree and outdegree at most one. Then $\mathcal{P}$ is a dipath or a directed circuit.
\end{lem}
\begin{proof}
	If $\mathcal{P}$ contains a directed circuit, then it must actually be a directed circuit, since any other path meeting it would increase the degree at some point. Thus we may assume that it does not. Let $R$ be an ascending sequence of paths in $\mathcal{P}$. Since $\mathcal{P}$ is rayless, there exist $P_1, \dots, P_k \in \mathcal{P}$ containing a final segment of the union of $R$. Choose these to minimize $k$.  We claim that $k \leq 1$. Indeed, by the degree condition no $P_i$ can leave $R$ in an inner point, so given paths $P_1$ and $P_2$ meeting $R$ we can obtain a path from their union by connecting them via segments of $R$, if neccessary. Thus we may find a maximal path $P$ in $\mathcal{P}$ by Zorn's Lemma. Now any path that is not a segment of $P$  can only attach at exactly one endpoint of $P$. This, however, would contradict the maximality of $P$.
\end{proof}

Given two dipath spaces $\mathcal{P}$ and $\mathcal{Q}$, we write $\mathcal{P} \triangle \mathcal{Q}$ for the space obtained from $\widehat{\mathcal{P} \cup \mathcal{Q}}$ by deleting $\mathcal{P} \cap \mathcal{Q}$.

For the rest of this section, fix a dipath space $\mathcal{G}$, sets $A,B \subseteq V(\mathcal{G})$ and a finite set $\mathcal{P}$ of disjoint $A$-$B$-paths in $\mathcal{G}$.

Call a trail $f: P \rightarrow \overline{\mathcal{G}}$ \emph{alternating} (with respect to $\mathcal{P}$) if it satisfies the following conditions:
\begin{enumerate}
	\item{It starts in $A \setminus V(\mathcal{P})$}.
	\item{$f$ has no nontrivial segment in $\mathcal{P}$ or in the inverse of any path not sharing a segment with $\mathcal{P}$.}
	\item{$|f^{-1}(v)| \leq 1$ for any $v \notin V(\mathcal{P}$.)}
	\item{For any $v \in V(\mathcal{P}) \cap V(\widehat{f})$ which is not the final vertex of $f$ there exists some nontrivial segment of $f$ containing $v$ whose inverse is contained in $\mathcal{P}$.}
	\item{Every path of $\mathcal{P}$ meets $\hat{f}$ in only finitely many segments.}
\end{enumerate}
\begin{lem} \label{symdiff}
Let $f$ be a trail alternating with respect to $\mathcal{P}$.
Then the space $\hat{\mathcal{P}} \triangle \hat{f}$ is rayless.
\end{lem}
\begin{proof}
Let $R=(P_\alpha)_{\alpha < \beta}$ be an ascending sequence of paths in $X=\hat{\mathcal{P}} \triangle \hat{f}$. We will show that there  exists some finite set of paths in $X$ whose union contains a final segment of $R$.
By the last condition for alternating trails we may assume w.l.o.g. that $R$ is completely contained in $\hat{\mathcal{P}}$ or $\hat{f}$.

 First assume that it is contained in $\hat{\mathcal{P}}$. Then there must exist a path $P \in \hat{\mathcal{P}}$ which contains a final segment of $R$, w.l.o.g.  $P$ with its final point deleted also meets $R$ cofinally. Since $R$ is contained in $\hat{\mathcal{P}}$, a final segment of $P$ has no segment in $\hat{f}$ and is thus contained in $X$.
 
 Now assume that $R$ is contained in $\hat{f}$. Since $\hat{f}$ is finitary, in particular there exist $P_1, \dots, P_k$ inducing a final segment $R$ in $\hat{f}$, w.l.o.g. $R$ meets each $P_i$ with its final point deleted cofinally. As before, final segments of each $P_i$ have no segment in $\hat{\mathcal{P}}$ and thus they are all contained in $X$.
\end{proof}
\begin{prop} \label{augment}
	If there is an alternating trail ending in $B \setminus V(\mathcal{P})$, there exists a set of disjoint $A$-$B$-paths $\mathcal{Q}$ with $|\mathcal{Q}|>|\mathcal{P}|$.
\end{prop}
\begin{proof}
	Let $T=\widehat{f} \triangle \widehat{\mathcal{P}}$.
	By \cref{walkfin,symdiff} $T$ is rayless. Let $A'$ be the set of initial points of $f$ and the paths of $\mathcal{P}$ and $B'$ their set of final points. Now define $\mathcal{Q}$ to be the set of components of $T$ meeting $A' \cup B'$.
	
	Note that $T$ has maximum indegree and outdegree one, exactly the points of $A'$ have indegree zero and exactly the points of $B'$ have outdegree zero. By \cref{pathcircuit} any element of $\mathcal{Q}$ is then a path starting in $A'$ and ending in $B'$. In particular, $|\mathcal{Q}| > |\mathcal{P}|$.
\end{proof}
\begin{prop} \label{separator}
	If there is no alternating trail ending in $B \setminus V(\mathcal{P})$, there exists a choice of one point from each element of $\mathcal{P}$ which meets every $A$-$B$-path.
\end{prop}
\begin{proof}
	For every $P \in \mathcal{P}$ let $x_P$ be the supremum in $P$ of all points $v$ such that there exists an alternating trail ending in $v$ and let $X$ be the set of all these points. We claim that $X$ meets every $A$-$B$-path. Let $S$ consist of the segment of each $P \in \mathcal{P}$ up to $x_P$.
	
	Assume for contradiction that there exists an $A$-$B$-path $Q$ avoiding $X$. Since $Q$ is not an alternating trail to a point of $B \setminus V(\mathcal{P})$ or $V(\mathcal{P}) \setminus S$, it meets $S$; let $y$ be its last point in it and $R$ the element of $\mathcal{P}$ containing $y$. Since $Q$ avoids $X$, $y \not= x_R$, so there exists a $z$ on $R$ after $y$ such that there is an alternating trail $f$ ending in $z$.
	Let $z'$ be the first point of $f$ on the segment of $R$ between $z$ and $y$ and let $f'$ be the concatenation of the segment of $f$ until $z'$ with the inverse of $R$ between $y$ and $z'$. Then $f'$ is an alternating trail ending in $y$. Since $f'$ meets $V(\mathcal{P})$ only in $S$ and $y$ is the last point of $Q$, the segment of $Q$ from $y$ can meet $f'$ only outside $V(\mathcal{P})$ and in $y$. If they only meet in $y$, let $f''$ be their concatenation. Otherwise, let $f''$ be the concatenation of $f'$ up to meeting $Q$ for the first time with $Q$ starting from that meeting point.
	
	Now $f''$ is alternating with respect to the set consisting of the segment of each $P$ up to $x_P$.
	But if $f''$ had a first vertex on the complementary segments, say on some path $P$, it would then be an alternating trail to that point, contradicting the choice of $x_{P}$.
\end{proof}
\begin{thm} \label{menger}
Let $\mathcal{G}$ be a dipath space, $A,B \subseteq V(\mathcal{G})$ and $k$ a natural number. Then either there exists a set of size less than $k$ meeting every $A$-$B$-path or a set of $k$ disjoint $A$-$B$ paths.
\end{thm}
\begin{proof}
Assume that there is no set of $k$ disjoint $A$-$B$-paths. Then there exists a set $\mathcal{P}$ of disjoint $A$-$B$-paths of maximal size. By \cref{augment} there is no trail alternating with respect to $\mathcal{P}$. But then there must be a set of size $|\mathcal{P}|<k$ meeting every $A$-$B$-path by \cref{separator}.
\end{proof}
While we have only looked at $A$-$B$-paths so far, the same proof will work if we replace all occurences of '$A$-$B$-paths' in \cref{menger} with 'paths from $A$ to $B$'.

\section{Counterexamples} \label{counterex}
Some natural-sounding alternate versions of \cref{menger} can easily be refuted by simple, known counterexamples. This is also shown in \cite{TV:Menger} with very similar examples originally given in \cite{DK:PCT} and \cite{Whyburn:Conn}.
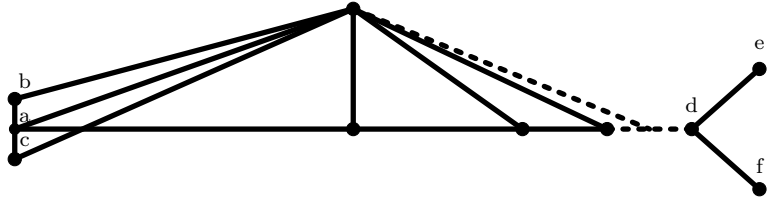
\begin{figure}
	\centering

\begin{tikzpicture}[line cap=round,line join=round,>=triangle 45,x=9cm,y=8cm]
\clip(-0.1,-0.2) rectangle (1.2,0.3);
\draw [line width=2pt] (0,0)-- (0.5,0);
\draw [line width=2pt] (0.5,0)-- (0.75,0);
\draw [line width=2pt] (0.75,0)-- (0.875,0);
\draw [line width=2pt,loosely dotted] (0.875,0)-- (1,0);
\draw [line width=2pt] (0,-0.05)-- (0,0);
\draw [line width=2pt] (0,0.05)-- (0,0);
\draw [line width=2pt] (0,0)-- (0.5,0.2);
\draw [line width=2pt] (0.5,0)-- (0.5,0.2);
\draw [line width=2pt] (0.75,0)-- (0.5,0.2);
\draw [line width=2pt] (1.1,0.1)-- (1,0);
\draw [line width=2pt] (1.1,-0.1)-- (1,0);
\draw [line width=2pt] (0.5,0.2)-- (0,-0.05);
\draw [line width=2pt] (0.5,0.2)-- (0,0.05);
\draw [line width=2pt] (0.5,0.2)-- (0.875,0);
\draw [line width=2pt,loosely dotted] (0.5,0.2)-- (0.9375,0);
\begin{scriptsize}
\draw [fill=black] (0,0) circle (2pt);
\draw[color=black] (0.015,0.02) node {a};5
\draw [fill=black] (0.5,0) circle (2.5pt);
\draw [fill=black] (0.75,0) circle (2.5pt);
\draw [fill=black] (1,0) circle (2.5pt);
\draw[color=black] (1,0.04) node {d};
\draw [fill=black] (0.875,0) circle (2.5pt);
\draw [fill=black] (0,-0.05) circle (2.5pt);
\draw[color=black] (0.015,-0.02) node {c};
\draw [fill=black] (0,0.05) circle (2.5pt);
\draw[color=black] (0.015,0.08) node {b};
\draw [fill=black] (0.5,0.2) circle (2.5pt);
\draw [fill=black] (1.1,0.1) circle (2.5pt);
\draw[color=black] (1.1,0.14) node {e};
\draw [fill=black] (1.1,-0.1) circle (2.5pt);
\draw[color=black] (1.1,-0.06) node {f};
\end{scriptsize}
\end{tikzpicture}
	\caption{A combined counterexample to three Menger variations}
	\label{figurecounter}
\end{figure}

\Cref{figurecounter} shows a combination of such examples: a dominated ray, where the first vertex $a$ and the end $d$ have two extra neighbors each ($b,c$ and $e,f$ respectively) and $b$ and $c$ are also adjacent to the dominating vertex.

In this example there is no single vertex except $a$ or $b$ meeting every path from $a$ to $b$, but there are not even two edge-disjoint paths between them. Thus a version of Menger for internally disjoint paths fails between $a$ and $b$, one for fans between $\{a,b\}$ and $d$ and one for edge-disjoint paths between $\{b,c\}$ and $\{e,f\}$.

One could also ask for a version for infinite cardinalities or even an Aharoni-Berger-type statement, but an example from \cite{TV:Menger}  shows that a cardinality version of \cref{menger} already fails for $\aleph_0$. For this consider a the space $[0,1]^2$ and take $A$ to be the points of the form $(0,\frac{1}{n})$ and $B$ those of the form $(1,\frac{n-1}{n})$. This example can also be modified to obtain a graph-like space. In an upcoming article (\cite{Heine:Ubi}), the author will build on this to obtain some negative ubiquity results.

\section*{Acknowledgement}
I thank Nathan Bowler for many helpful discussions and comments.

\clearpage
\printbibliography
\end{document}